\documentclass[a4paper,12pt]{article}
\usepackage{amsmath, amsthm, amstext}
\usepackage{amssymb, array, amsfonts}
\usepackage[cp1251,cp866]{inputenc}
\usepackage[english]{babel}
\usepackage{bbding}
\inputencoding{cp1251}

\renewcommand{\l}{\left}
\renewcommand{\r}{\right}

\def \C{\mathbb{C}}
\def \N{\mathbb{N}}

\newcommand{\beq}{\begin{equation}}
\newcommand{\eeq}{\end{equation}}

\theoremstyle{plain}
\newtheorem{theorem}{\bf Theorem}
\newtheorem{lemma}[theorem]{\bf Lemma}

\theoremstyle{definition}

\theoremstyle{remark}

\renewcommand{\le}{\leqslant}
\renewcommand{\ge}{\geqslant}
\renewcommand{\mod}{\mathop{\mathrm{mod}}\nolimits}
\newcommand{\tr}{\mathop{\mathrm{tr}}\nolimits}
\newcommand{\diag}{\mathop{\mathrm{diag}}\nolimits}

\title{A Hilbert-Schmidt analog of Huaxin Lin's Theorem}

\author{N.~Filonov$^*$, I.~Kachkovskiy\thanks{St.~Petersburg Department of  V.~A.~Steklov Institute of Mathematics, Russian Academy of  Sciences.}}
\date{}
\begin{document}
\maketitle

The paper is devoted to the following question: 
consider two self-adjoint $n\times n$-matrices 
$H_1,H_2$, $\|H_1\|\le 1$, $\|H_2\|\le 1$, 
such that their commutator $[H_1,H_2]$ is small in some sense. 
Do there exist such self-adjoint commuting matrices $A_1,A_2$, 
such that $A_i$ is close to $H_i$, $i=1,2$?
The answer to this question is positive if the smallness 
is considered with respect to the operator norm. 
The following result was established by Huaxin Lin in \cite{L}: 
if $\|[H_1,H_2]\|=\delta$, then we can choose $A_i$ 
such that $\|H_i-A_i\|\le C(\delta)$, $i=1,2$, 
where $C(\delta)\to 0$ as $\delta\to 0$. 
Notice that $C(\delta)$ does not depend on $n$. 
The proof was simplified by 
Friis and R{\o}rdam in \cite{FR}. 
A quantitative version of the result 
with $C(\delta)=E(1/\delta)\delta^{1/5}$, 
where $E(x)$ grows slower than any power of $x$, 
was recently established by Hastings in \cite{H}.

We are interested in the same question, 
but with respect to the normalized Hilbert-Schmidt norm: 
for $A\in M_n(\C)$, $A=\{a_{ij}\}_{i,j=1}^n$, let
$$
\|A\|_{\tr}^2=\frac{1}{n}\sum\limits_{i,j=1}^n |a_{ij}|^2.
$$
An analog of Lin's theorem for this norm was established in 
\cite{H1,H2} (in a big generality, and also for the case of $n$ operators) and independently in \cite{FS}. 
A quantitative version with $C(\delta)=12\delta^{1/6}$, 
where $\delta=\|[H_1,H_2]\|_{\tr}$, was obtained by Glebsky in \cite{Gl}. 
In the present paper, we use the same ideas to prove a similar result 
with $C(\delta)=2\delta^{1/4}$, see Theorem \ref{main} below. We also consider the case of $n$ operators in a way similar to \cite{Gl}, see Theorem \ref{multi}.

\begin{lemma}
\label{l1}
Let $-1\le \lambda_1\le\ldots\le \lambda_n\le 1$. 
Then for any $k,m\in\N$ there exists a partition
$$
\{1,\ldots,n\}=J\cup\bigcup_{a=-m}^m L_a
$$
such that
\begin{enumerate}
\item
$\#J\le \frac{n}{k}$.
\item
$|\lambda_i-\lambda_j|<\frac{1}{m}$, $i,j\in L_a$.
\item
$|\lambda_i-\lambda_j|\ge \frac{1}{km}$, $i\in L_a$, $j\in L_b$, $a\neq b$.
\end{enumerate}
\end{lemma}

\begin{proof}
Consider the following partition $\{1,\ldots,n\}=\bigcup_{-km}^{km-1} I_j$:
$$
I_j=\l\{l\colon \lambda_l\in \l(\frac{j}{km};\frac{j+1}{km}\r]\r\},\, j=-km+1,\ldots,km-1;\quad 
I_{-km}=\l\{l\colon \lambda_l\in \l[-1;-1+\frac{1}{km}\r]\r\}.
$$
Let us combine $I_j$ with $j\equiv r(\mod\,k)$ into $J_r$:
$$
J_{r}=\bigcup\limits_{a=-m}^{m-1}I_{ak+r},\quad r=0,1,\ldots,k-1.
$$
Obviously, $\bigcup_{r=0}^{k-1} J_r=\{1,\ldots,n\}$. By the Dirichlet principle, there exists 
such $r_0$ that $\#J_{r_0}\le \frac{n}{k}$. We set $J=J_{r_0}$. As for $\{L_a\}$, let
$$
L_a=\bigcup\limits_{(a-1)k+r_0<j<ak+r_0}I_j,\quad a=-m,\ldots,m.
$$
Property 1 follows from the construction of $J=J_{r_0}$. Furthermore, every 
interval $J_r$ consists of $k-1$ subsequent intervals $I_l$, so, if $i,j\in J_r$, then
$$
|\lambda_i-\lambda_j|\le \frac{k-1}{km}< \frac{1}{m},
$$
which implies Property 2. Finally, two intervals $L_a$ and $L_b$, $a\neq b$, are 
separated by one of $I_{ak+r_0}$, and hence Property 3 is true.
\end{proof}

\begin{theorem}
\label{main}
Let 
$$
H_j\in M_{n}(\C), \quad \|H_j\|\le 1, \quad H_j=H_j^*, \ j=1,2.
$$ 
Let $\|[H_1,H_2]\|_{\tr}=\delta\le 1/16$. 
Then there exist $A_1,A_2\in M_n(\C)$ 
such that 
$$
\|A_j\|\le 1, \quad A_j=A_j^*, \quad
\|H_j-A_j\|_{\tr}\le 2\,\delta^{1/4},\ j=1,2,
$$ 
and $[A_1,A_2]=0$. 
In addition, $[H_1,A_1]=0$.
\end{theorem}

\begin{proof}
We can choose such a basis in $\C^n$ that 
$$
H_1=\diag(\lambda_1,\ldots,\lambda_n),
\quad -1\le \lambda_1\le\ldots\le \lambda_n\le 1.
$$ 
Let $k,m\in\N$ be chosen later. 
Consider the corresponding partition 
$\{1,\ldots,n\}=J\cup\bigcup_{a=-m}^m L_a$ from Lemma \ref{l1}. 
We set
$$
A_1=\diag(\mu_1,\ldots,\mu_n),
$$
where $\mu_j=\lambda_j$ for $j\in J$,
and for all $j\in L_a$ with $a$ fixed the number $\mu_j$ 
is the center of the interval of possible values of $\lambda_i$, $i\in L_a$.
Obviously, $\|A_1\|\le 1$. 
Property 2 from Lemma \ref{l1} implies 
$|\lambda_j-\mu_j|\le \frac{1}{2m}$ for all $j$. 
Hence,
\begin{equation}
\label{0}
\|A_1-H_1\|_{\tr}^2=\frac{1}{n}\sum\limits_{j=1}^n |\mu_j-\lambda_j|^2\le \frac{1}{4m^2}.
\end{equation}
In the chosen basis, let $H_2=\{h_{ij}\}_{i,j=1}^n$, so $[H_1,H_2]_{ij}=(\lambda_i-\lambda_j)h_{ij}$. Then
\begin{equation}
\label{1}
\sum\limits_{i,j=1}^n|\lambda_i-\lambda_j|^2|h_{ij}|^2=n\delta^2.
\end{equation}
We construct $A_2=\{a_{ij}\}_{i,j=1}^n$ as following: 
$$
a_{ij}=\begin{cases} h_{ij},&\exists b\colon i,j\in L_b;\\ 0,&\mbox{otherwise}.\end{cases}
$$
$A_2$ is a block diagonal matrix. 
The norm of each block does not exceed $\|H_2\|$, 
because it is a part of $H_2$, and so 
$\|A_2\|\le \|H_2\|\le 1$. 
In each block $A_1$ is scalar, which follows $[A_1,A_2]=0$. 
Also, $A_2=A_2^*$. 
Let us estimate the difference between $A_2$ and $H_2$.
\begin{equation}
\label{2}
n\|A_2-H_2\|^2_{\tr}\le 
\sum\limits_{a\neq b}\sum\limits_{i\in L_a}\sum\limits_{j\in L_b}|h_{ij}|^2+
2\sum\limits_{i\in J}\sum\limits_{j=1}^n|h_{ij}|^2.
\end{equation}
In the second sum we used the fact that $h_{ij}=\overline{h}_{ji}$. 
The first sum can be estimated 
using \eqref{1} and property 3 from Lemma \ref{l1}:
\begin{equation}
\label{3}
\sum\limits_{a\neq b}\sum\limits_{i\in L_a}\sum\limits_{j\in L_b}|h_{ij}|^2
\le k^2 m^2 \sum\limits_{a\neq b}\sum\limits_{i\in L_a}
\sum\limits_{j\in L_b}|\lambda_i-\lambda_j|^2|h_{ij}|^2
\le n\delta^2k^2m^2.
\end{equation}
To estimate the second sum, consider a matrix $\widetilde{H}=\{\tilde{h}_{ij}\}_{i,j=1}^n$,
$$
\tilde{h}_{ij}=
\begin{cases} 
h_{ij},&i\in J;\\ 0,&\mbox{otherwise},
\end{cases}
$$
and a matrix
$$
P = \diag (p_1, \ldots, p_n), \quad
\begin{cases}
p_j = 1,& j \in J,\\
p_j = 0,& j \not\in J .
\end{cases}
$$
Clearly, 
$\widetilde{H}= P H_2$ and $\|\widetilde{H}\|\le \|H_2\| \le 1$. 
Further, 
\begin{equation}
\label{4}
\sum\limits_{i\in J}\sum\limits_{j=1}^n|h_{ij}|^2
= \tr (P H_2^2 P) \le \tr P \|H_2\|^2 \le \# J \le \frac{n}{k}.
\end{equation}
Combining the inequalities \eqref{2} -- \eqref{4}, we obtain
$$
\|A_2-H_2\|_{\tr}^2\le \delta^2 k^2 m^2+\frac{2}{k}.
$$
Finally, we set
$$
k = \left[\frac{2}{\delta^{1/2}}\right],
\quad m = \left[\frac{1}{2\delta^{1/4}}\right].
$$
Then
\beq
\label{a1h1}
\|A_1-H_1\|_{\tr}\le \frac{1}{2m} \le 2\delta^{1/4},
\eeq
and
\beq
\label{a2h2}
\|A_2-H_2\|_{\tr}\le \sqrt{\delta^{1/2} + \frac{2}{k}} 
\le \sqrt{3}\, \delta^{1/4},
\eeq
where we used \eqref{0}, the fact that $2\delta^{1/4} \le 1$,
and the inequality $[x]^{-1} \le 2 x^{-1}$ for $x\ge 1$.
\end{proof}
It is possible to rewrite Theorem 4 from \cite{Gl} in the following form:
\begin{theorem}
\label{multi}
Let $m\ge 3$,
$$
H_j\in M_{n}(\C), \quad \|H_j\|\le 1, \quad H_j=H_j^*, \ j=1,\ldots,m.
$$ 
Let $\|[H_i,H_j]\|_{\tr}\le \delta$, $i,j=1,\ldots,m$, let also 
\beq
\label{delta_cond}
\delta\le \frac{1}{16^{2\cdot4^{m-2}}}.
\eeq
Then there exist $A_i\in M_n(\C)$, $i=1,\ldots,m$, 
such that 
$$
\|A_j\|\le 1, \quad A_j=A_j^*, \quad
\|H_j-A_j\|_{\tr}\le 5\delta^{1/4^{m-1}},\ j=1,\ldots,m,
$$ 
and 
$$
[A_i,A_j]=0,\quad i,j=1,\ldots,m.
$$
In addition, $[H_1,A_1]=0$.
\end{theorem}
\begin{proof}
The scheme from Theorem \ref{main} can be applied simultaneously to the pairs $(H_1,H_j)$, $j=2,\ldots,m$. We denote the resulting operators by $\widetilde{H}_i$, $i=1,\ldots,m$. If $\delta\le 1/16$, then, by \eqref{a1h1}--\eqref{a2h2},
$$
\|H_1-\widetilde{H}_1\|_{\tr}\le 2\delta^{1/4};\quad\|H_i-\widetilde{H}_i\|_{\tr}\le \sqrt{3}\delta^{1/4},\quad i=2,\ldots,m.
$$
Let us estimate the commutators of $\widetilde{H}_i$:
$$
\|[\widetilde{H}_i,\widetilde{H}_j]-[H_i,H_j]\|_{\tr}\le 
$$
$$
\le \|(\widetilde{H}_i-H_i)\widetilde{H}_j\|_{\tr}+
\|H_i(\widetilde{H}_j-H_j)\|_{\tr}+\|(H_j-\widetilde{H}_j)H_i\|_{\tr}+\|\widetilde{H}_j(H_i-\widetilde{H}_i)\|_{\tr}\le 4\sqrt{3}\delta^{1/4},
$$
where we again used \eqref{a2h2} and the fact that $\|AB\|_{\tr}\le \|A\|\|B\|_{\tr}$. So,
$$
\|[\widetilde{H}_i,\widetilde{H}_j]\|_{\tr}\le (4\sqrt{3}+\delta^{3/4})\delta^{1/4}\le 8\delta^{1/4}
$$
and
$$
[\widetilde{H}_1,\widetilde{H}_i]=0,\quad i=2,\ldots,m.
$$
The last relation will remain true if we again apply the scheme from Theorem \ref{main} to the pairs $(\widetilde{H}_2,\widetilde{H}_j)$, $j=3,\ldots,m$,
because the scheme preserves common invariant subspaces. Hence, we
can proceed with $m-1$ such iterations and obtain a set of $m$ commuting operators $A_1,\ldots,A_m$. Let us estimate the difference between $A_i$ and $H_i$ and find the conditions on $\delta$.

We denote $\delta$ from the statement of Theorem by $\delta_1$. After $i$-th iteration, $\delta_{i}$ is replaced with $\delta_{i+1}=8\delta_i^{1/4}$. This gives
$$
\delta_{i}=8^{1+1/4+1/16+\ldots+1/4^{i-1}}\delta^{1/4^{i-1}}\le 16\delta^{1/4^{i-1}}.
$$
The sequence $\{\delta_i\}$ is increasing. Condition \eqref{delta_cond} implies $\delta_{m-1}\le 1/16$ and, consequently, $\delta_i\le 1/16$, $i=1,\ldots,m-1.$ We now see that Theorem \ref{main} is applicable on every step.

Finally, consider the difference between $A_i$ and $H_i$. On $i$-th iteration, we "correct" the operators by 
$$
2\delta_i^{1/4}=\frac14\delta_{i+1}\le 4\delta^{1/4^i}.
$$
So,
$$
\|H_i-A_i\|_{\tr}\le 2(\delta_1^{1/4}+\delta_2^{1/4}+\ldots+\delta_{m-1}^{1/4})\le
$$
$$
\le 4(\delta^{1/4^{m-1}}+\delta^{1/4^{m-2}}+\ldots+\delta^{1/4})
\le 4\gamma(1+\gamma^4+\gamma^{16}+\ldots)\le 5\gamma,
$$
where $\gamma=\delta^{1/4^{m-1}}\le1/4$.
\end{proof}

\end{document}